\documentclass[a4paper]{amsart}
\pdfoutput=1
\usepackage{tikz,pgfplots}
\usepackage[utf8]{inputenc} 
\usepackage{
	amsmath,amssymb, 
	graphicx, 
	mathtools,
	bm,
	todonotes,
	hypbmsec,
	float,
	subcaption,
	caption,
	adjustbox
}
\usepackage[shortlabels]{enumitem}

\makeatletter
\@namedef{subjclassname@2020}{\textup{2020} Mathematics Subject Classification}
\makeatother

\RequirePackage[colorlinks,citecolor=blue,urlcolor=blue]{hyperref}
\usepackage[margin=3cm]{geometry}

\usepackage{etoolbox}

\DeclarePairedDelimiterX\abs[1]\lvert\rvert{
	\ifblank{#1}{\:\cdot\:}{#1}
}
\DeclarePairedDelimiterX\norm[1]\lVert\rVert{ 
	\ifblank{#1}{\:\cdot\:}{#1}
}
\DeclarePairedDelimiterX{\inner}[2]{\langle}{\rangle}{ 
	\ifblank{#1}{\:\cdot\:}{#1},\ifblank{#2}{\:\cdot\:}{#2}
}

\providecommand\given{}
\DeclarePairedDelimiterX\Set[1]{\lbrace}{\rbrace}{
	\renewcommand\given{\SetSymbol[\delimsize]}
	#1
}
\DeclarePairedDelimiterXPP\Prob[1]{\mathbb{P}}(){}{
	\renewcommand\given{\nonscript\:\delimsize\vert\nonscript\:
		\mathopen{}}
	#1}
\DeclarePairedDelimiterXPP\Var[1]{\text{Var}}(){}{
	\renewcommand\given{\nonscript\:\delimsize\vert\nonscript\:
		\mathopen{}}
	#1}
\DeclarePairedDelimiterXPP\Mean[1]{\mathbb{E}}[]{}{
	\renewcommand\given{\nonscript\:\delimsize\vert\nonscript\:
		\mathopen{}}
	#1}

\newlist{exercise}{enumerate}{3}
\setlist[exercise]{wide, labelwidth=!, labelindent=0pt, label=\textbf{(\alph*)}}

\theoremstyle{plain}
\newtheorem{theorem}{Theorem}
\newtheorem{conjecture}[theorem]{Conjecture}

\newtheorem{corollary}[theorem]{Corollary}
\newtheorem{proposition}[theorem]{Proposition}
\newtheorem{lemma}[theorem]{Lemma}
\newtheorem{assumption}{Assumption}

\newtheorem{remark}[theorem]{Remark}
\newtheorem{example}[theorem]{Example}

\newcommand\N{\mathbb{N}}

\newcommand\R{\mathbb{R}}

\usepackage[mathscr]{euscript} 

\DeclarePairedDelimiterX\lrangle[1]\langle\rangle{
	\ifblank{#1}{\:\cdot\:}{#1}
}

\newcommand\idd{\,\mathrm{d}} 
\newcommand\dd{\mathrm{d}} 
\newcommand\ff{\mathcal{F}}

\newcommand\cp{\mathcal{C}}

\newcommand\D{\mathcal{D}}

\newcommand{\n}{{(n)}}
\newcommand{\p}{\mathbb P}

\newcommand{\convd}{\overset{d}{\rightarrow}}
\newcommand{\cip}{\overset{\p}{\rightarrow}}

\newcommand{\ind}[1]{\mbox{\rm\large 1}_{{#1}}}

\begin{document}
\title[L\'evy processes conditioned to stay in a half-space]{L\'evy processes conditioned to stay in a half-space with applications to directional extremes}
\author[J.\ Ivanovs]{Jevgenijs Ivanovs}
\author[J.\ D.\ Thøstesen]{Jakob D.\ Thøstesen}
\begin{abstract}
This paper provides a multivariate extension of Bertoin's pathwise construction of a L\'evy process conditioned to stay positive/negative.
Thus obtained processes conditioned to stay in half-spaces are closely related to the original process on a compact time interval seen from its directional extremal points.
In the case of a correlated Brownian motion the law of the conditioned process is obtained by a linear transformation of a standard Brownian motion and an independent Bessel-3 process.
Further motivation is provided by a limit theorem corresponding to zooming in on a L\'evy process with a Brownian part at the point of its directional infimum.
Applications to zooming in at the point furthest from the origin are envisaged.
\end{abstract}
\maketitle

\section{Introduction}

There are multiple examples of conditioning a univariate L\'evy process in some limiting sense, which alternatively can be described by Doob $h$-transforms, see~\cite{bertoin_splitting,doring_watson_weissmann,doring_weissmann} and references therein. Most often the focus is on establishing properties directly related to these conditional processes.
The case of conditioning to stay positive/negative is special in the sense that it is intimately related to the post- and pre-infimum processes~\cite{bertoin_splitting}, leading to various important applications.
Further links to path decomposition results can be found in~\cite{duquesne}. 

Local behavior of a univariate L\'evy process at its extremal points is studied in~\cite{iva_zooming}, see also~\cite{bertoin_splitting} for a self-similar case and~\cite{AGP} for a linear Brownian motion.
It is shown that zooming in at the point of infimum results in a pair of processes obtained from the underlying self-similar L\'evy process conditioned to stay positive and negative.
Further applications of this theory in the setting of high-frequency statistics include estimation of threshold exceedance in~\cite{bis_iva} and optimal estimation of extremes in~\cite{iva_pod}. 
Bertoin's pathwise construction of conditioned processes in~\cite{bertoin_splitting} plays a fundamental role in these works. For yet another application see~\cite{asmussen2018discretization} studying the discretization error in the two-sided Skorokhod reflection map.

In this work we extend Bertoin's construction to the multivariate setting to define a L\'evy process conditioned to stay in a half-space specified by some normal vector $\eta\neq 0$, see \S\ref{sec:representation}. Importantly, the link to post- and pre-extremum processes is preserved, where extrema are understood with respect to the direction~$\eta$. 
Furthermore, in \S\ref{sec:application} we establish an associated invariance principle which, in particular, yields a limit result when zooming in on a L\'evy process at the point of directional extremum. This is achieved via a short and direct argument relying on the path-wise construction. 
Applications of this result to high frequency statistics and the study of discretization errors in problems related to directional extrema and exceedance are anticipated.


In the multivariate case we have a continuum of possible directions, and the effect of linear transformations is studied in \S\ref{sec:linear}. It is shown that conditioning with respect to any direction $\eta$ can be reduced to, say, conditioning an appropriately rotated process so that its first component stays positive. 
Furthermore, we provide a simple expression for the conditioned correlated Brownian motion in terms of a certain linear transformation of independent standard Brownian motions and a Bessel-3 process.
In \S\ref{sec:law} we present the semigroup of the conditioned process in the general case, which turns out to have an intuitive structure. In \S\ref{sec:Feller} we utilize the arguments and insights from~\cite{chaumont_doney} to establish some important properties of the conditioned process. This leads to a natural definition of the respective Feller process started from an arbitrary point in the closed half-space.

We have attempted to present the multivariate theory in a streamlined and concise form, while emphasizing the main novelties stemming from the multivariate setting. Finally, in \S\ref{sec:further_applications} we state a conjecture related to the local behavior at the point furthest from the origin, which hints at even greater application potential of the multivariate theory.

\section{Preliminaries}
Fix an integer $d\geq 1$ and let $\D$ denote the space of càdlàg functions $\omega:\R\to\R^d\cup\Set{\dagger}$, where $\dagger$ is an isolated absorbing state.
As usual we equip the path space $\D$ with the Skorokhod topology and let $\ff$ denote the Borel $\sigma$-field. Furthermore we denote the coordinate process by $X=(X_t)$ and  its natural completed filtration by $(\ff_t)$.
Unless stated otherwise we work with a subclass of processes satisfying $X_t=0$ for $t<0$, and let $\zeta:=\inf\Set{t\geq 0 \given X_t=\dagger}\in[0,\infty]$ be the lifetime. 

\subsection{Directional infimum}
We shall consider a fixed vector $\eta\in\R^d\setminus\Set{0}$ and the respective open and closed half-spaces 
\[S:=\Set{x\in\R^d\given\inner{x}{\eta}>0}, \qquad \overline S:=\Set{x\in\R^d\given\inner{x}{\eta}\geq 0};\]
for ease of notation we omit $\eta$ here and in the following.
The \emph{projected process} is defined by
\[Z_t:=\inner{X_t}{\eta}\in\R\cup\{\dagger\},\] 
where $\inner{\dagger}{\eta}=\dagger$ by convention.

Assume for a moment that the lifetime is finite and strictly positive, $\zeta\in(0,\infty)$.
Consider the directional infimum $\underline Z:=\inf\Set{Z_t\given t\geq 0}$ and the respective (last) time
\[\tau:=\sup\Set{t\geq 0\given Z_t\wedge Z_{t-}=\underline Z}\in[0,\zeta],\]
where $z\wedge\dagger=z$.
Letting $\underline X:=X_\tau\ind{\Set{Z_\tau\leq Z_{\tau-}}}+X_{\tau-}\ind{\Set{Z_\tau> Z_{\tau-}}}$ be the position of $X$ at the time of directional infimum, 
we define the (directional) post-infimum and reversed pre-infimum processes by
\begin{align*}
	\underrightarrow{X}_t:=\begin{cases}
		X_{\tau+t}-\underline X &\text{if }0\leq t<\zeta-\tau \\
		\dagger &\text{if }t\geq \zeta-\tau,
	\end{cases} \qquad
	\underleftarrow{X}_t:=\begin{cases}
		X_{(\tau-t)-}-\underline X &\text{if }0\leq t<\tau \\
		\dagger &\text{if }t\geq \tau,
	\end{cases}
\end{align*}
see also Figure~\ref{fig:illustration} for a schematic illustration.
 According to the above convention we set $\underrightarrow{X}_t=\underleftarrow{X}_t=0$ for $t<0$.
Note that $\underrightarrow{X}_t=\dagger$ for $t\geq 0$ if $\tau=\zeta$, and similarly $\underleftarrow{X}_t=\dagger$ for $t\geq 0$ if $\tau=0$.
The pair of processes $(\underleftarrow{X},\underrightarrow{X})$ is a representation of the process $X$ seen from the time-space point $(\tau,\underline X)$.
Alternatively, we could have defined a proper two-sided process. 

\subsection{L\'evy processes}
Throughout this paper $\p$ will be a probability measure on $(\D,\ff)$ such that $X$ is a $d$-dimensional Lévy process with infinite lifetime.
We write $X:\p$ when there is a need to specify the law of $X$ explicitly.
For a deterministic $T\in(0,\infty)$ the process $X:\p$ sent to $\dagger$ at  $T$ is denoted by $X:\p^T$, and in particular $\p^T(\zeta=T)=1$.
By default we work with $\p$ if no law is mentioned explicitly.
The L\'evy measure of $X$ is denoted by $\Pi(\dd x)$.
Additional notation will be introduced in the following when required.  

Throughout this paper we assume (the excluded case is simple but somewhat cumbersome):
\begin{assumption}\label{as:notCPP}
	For the chosen direction $\eta$ the projected process $Z$ is not a compound Poisson process. 
\end{assumption}
Under Assumption~\eqref{as:notCPP} it is well known that the process $Z:\p^T$ achieves its infimum once only (at the time~$\tau$) a.s.
This means that $\underrightarrow{X}$ and $\underleftarrow{X}$ are inside the open half-space~$S$ for strictly positive times preceding~$\zeta$.
Our next result shows that $X$ cannot jump perpendicularly to~$\eta$ at~$\tau$, see Figure~\ref{fig:illustration}, 
and so $\underrightarrow{X}_0$ and $\underleftarrow{X}_0$ are either at the origin or inside $S\cup\Set{\dagger}$. For the definition of regular/irregular points we refer to~\cite[p.\ 104]{bertoin}.
\begin{lemma}\label{lem:trichotomy}
	The following trichotomy holds with respect to the projected process $Z:\p$.
	\begin{itemize}
		\item[$(\updownarrow)$] If $0$ is regular for $(-\infty,0)$ and for $(0,\infty)$ then $\underleftarrow{X}_0=\underrightarrow{X}_0=0$\quad $\p^T$-a.s.
		\item[$(\uparrow)$] If $0$ is irregular for $(-\infty,0)$ then $\underleftarrow{X}_0\in S\cup\Set{\dagger}$ and $\underrightarrow{X}_0=0$\quad $\p^T$-a.s.
		\item[$(\downarrow)$] If $0$ is irregular for $(0,\infty)$ then $\underleftarrow{X}_0=0$ and $\underrightarrow{X}_0\in S\cup\Set{\dagger}$\quad $\p^T$-a.s.
	\end{itemize}
\end{lemma}
\begin{proof}
The latter two statements are easy and follow from the univariate case. 
Suppose instead that $0$ is regular for both half-lines, in which case $\p^T(\tau\in\Set{0,T})=0$. We may choose a sequence $(T_n)$ of stopping times, ranging over all jump epochs of $X$. 
Applying the strong Markov property yields $\p^T(Z_{T_n}=\underline{Z})=0$ since $Z$ is regular for $(-\infty,0)$. Thus, if $X$ jumps at $\tau$ then $Z_\tau>\underline{Z}$ $\p^T$-a.s. 
The same argument applied to the time reversed process $(X_T-X_{(T-t)-})$ having the law of $X:\p^T$ shows that $Z_{\tau-}>\underline{Z}$  if $X$ jumps at $\tau$ $\p^T$-a.s.; here we employ regularity for $(0,\infty)$.
We conclude that $X$ is $\p^T$-a.s.\ continuous at $\tau$ and this proves the statement.
\end{proof}

\begin{figure}[h!]
\begin{centering}
\begin{tikzpicture}
\draw[gray] (-2,1) -- (2,-1);
\draw[gray,->] (0,0) -- (0.5,1) node[anchor=east]{$\eta$};
\draw[thick] plot [smooth, tension=1] coordinates { (-0.9,2) (-0.8,1.6)  (0.2,1.5) (-0.3,1.9) (-0.5,1)};
\draw[dashed,thick] (-0.5,1)--(0,0);
\draw[thick,->] plot [smooth, tension=1] coordinates {(0,0)  (1,0.5) (1.1,-0.2) (0.7,0.6) (1.3,1.1)};
\filldraw (0,0) circle (2pt);
\draw(-0.5,1) circle (2pt);
\draw (0,0) node[anchor=north east]{$(0,0)$};
\end{tikzpicture}
\qquad
\begin{tikzpicture}
\draw[gray] (-2,1) -- (2,-1);
\draw[gray,->] (0,0) -- (0.5,1) node[anchor=west]{$\eta$};
\draw[thick,xshift=0.5cm,yshift=-1cm] plot [smooth, tension=1] coordinates { (-0.9,2) (-0.8,1.6)  (0.2,1.5) (-0.3,1.9) (-0.5,1)};
\draw[thick,->,xshift=0.7cm,yshift=-0.1cm] plot [smooth, tension=1] coordinates {(0,0)  (1,0.5) (1.1,-0.2) (0.7,0.6) (1.3,1.1)};
\draw[dashed,thick] (0,0)--(0.7,-0.1);
\draw (0,0) circle (2pt);
\filldraw(0.7,-0.1) circle (2pt);
\draw (0,0) node[anchor=north east]{$(0,0)$};
\end{tikzpicture}
\qquad
\begin{tikzpicture}
\draw[gray] (-2,1) -- (2,-1);
\draw[gray,->,xshift=1cm,yshift=-0.5cm] (0,0) -- (0.5,1) node[anchor=west]{$\eta$};
\draw[thick,xshift=0.5cm,yshift=-1cm] plot [smooth, tension=1] coordinates { (-0.9,2) (-0.8,1.6)  (0.2,1.5) (-0.3,1.9) (-0.5,1)};
\draw[thick,->,xshift=1cm,yshift=-0.5cm] plot [smooth, tension=1] coordinates {(0,0)  (1,0.5) (1.1,-0.2) (0.7,0.6) (1.3,1.1)};
\draw[dashed,thick] (0,0)--(1,-0.5);
\draw (0,0) circle (2pt);
\filldraw(1,-0.5) circle (2pt);
\end{tikzpicture}
\end{centering}
\caption{Schematic illustration of the process in $\R^2$ seen from its directional infimum:  $(\uparrow)$ jump into $\eta$-minimum (left), $(\downarrow)$ jump out of $\eta$-infimum (center) and an impossible case (right).}
\label{fig:illustration}
\end{figure}
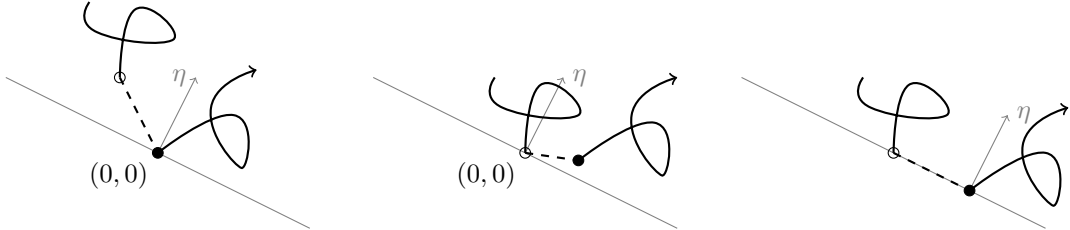


\section{The fundamental representation and the limit object}\label{sec:representation} 
We start with a fundamental representation of the law of the pair $(\underleftarrow{X},\underrightarrow{X}):\p^T$, which extends a univariate construction by Bertoin~\cite{bertoin_splitting} based, in turn, on an implicit identity for random walks~appearing in~\cite[Lem.\ XII.8.3]{feller}. 
Our representation is in terms of time-changed stochastic integrals, since the construction in~\cite{bertoin_splitting} in terms of the local time at 0 does not have a simple analogue in the multivariate setting.

Consider the non-killed process $X$ and let $\tilde X_t:=X_{(-t)-}$ be its time-reversal, which is a process with stationary and independent increments for negative times.
Define two $(\ff_t)$-adapted càdlàg processes $Y^\pm$ by
\begin{equation*}
	Y_t^+:=-\int_{[-t,0]}\ind{\Set{\inner{\tilde X_{s-}}{\eta}>0}}\idd\tilde X_s,\quad Y_t^-:=-\int_{[-t,0]}\ind{\Set{\inner{\tilde X_{s-}}{\eta}\leq0}}\idd\tilde X_s\quad\text{for }t\geq 0,
\end{equation*}
and $Y^\pm_\infty:=\dagger$. These stochastic integrals can be understood intuitively as $\int_0^t\ind{\Set{\inner{X_s}{\eta}>0}}\idd X_s$ and $\int_0^t\ind{\Set{\inner{X_s}{\eta}\leq0}}\idd X_s$, where the integrands are not predictable.

The cumulative times when $X$ is and is not in $S$ are denoted by $A^+$ and $A^-$ respectively. That is,
\begin{equation*}
	A_t^+:=\int_0^t\ind{\Set{\inner{X_s}{\eta}>0}}\idd s,\quad A_t^-:=\int_0^t\ind{\Set{\inner{X_s}{\eta}\leq0}}\idd s\quad\text{for }t\geq0.
\end{equation*}
Consider now the right-continuous inverses $\alpha^\pm_t:=\inf\Set{s\geq0\given A^\pm_s>t}$ of $A^\pm$, and define
\begin{equation*}
	X^\uparrow_t:=Y^+_{\alpha^+_t},\quad X^\downarrow_t:=Y^-_{\alpha^-_t}\quad\text{for }t\geq0.
\end{equation*}
The processes $X^\uparrow$ and $X^\downarrow$ under $\p^T$ are obtained by killing $X^\uparrow$ and $X^\downarrow$ at the times $A^+_T$ and $A^-_T$ under $\p$, respectively.
The times $A^\pm_T$ are non-decreasing in $T$, which results in longer lifetimes $\zeta^\uparrow$ and $\zeta^\downarrow$ for larger time horizons~$T$.

\begin{theorem}\label{thm:representation}
	Under $\p^T$ for $T\in(0,\infty)$ there is the following identity in law:
	\[(\underleftarrow{X},\underrightarrow{X})\stackrel{d}{=}(-X^\downarrow,X^\uparrow),\]
	where $-\dagger=\dagger$ by convention.
\end{theorem}
\begin{proof}
The proof is based on a random walk approximation and exchangeability of increments as in the one-dimensional cases of~\cite{bertoin_splitting}; it is deferred to \S\ref{sec:representation_proof}.
\end{proof}

Importantly, the above construction of the pair $(X^\downarrow,X^\uparrow):\p^T$ depends on $T$ via the killing times $A^\pm_T$ alone.
In particular, for $0<T_1<T_2$ the paths of $X^\uparrow:\p^{T_1}$ and $X^\uparrow:\p^{T_2}$ coincide up to the time $A^+_{T_1}$ when the former is sent to~$\dagger$, whereas the latter is killed at $A^+_{T_2}\geq A^+_{T_1}$.
It is convenient to think of growing the paths as $T$ increases. As $T\to\infty$ we obtain $(X^\downarrow,X^\uparrow)$.

\begin{corollary}\label{cor:object}
It holds that 
\[(\underleftarrow{X},\underrightarrow{X}):\p^T\convd (-X^\downarrow,X^\uparrow),\qquad T\to\infty.\]
\end{corollary}
It is noted that the above weak convergence statement can be strengthened, see~\cite[Cor.\ 3.2]{bertoin_splitting}, but we prefer using Theorem~\ref{thm:representation} directly when needed.
The pair $(X^\downarrow,X^\uparrow)$ is our main object of interest. 
According to Corollary~\ref{cor:object}, the process $X^\uparrow:\p$ can be called a limiting post-infimum process.
In analogy to the univariate case we instead call it \emph{$X$ conditioned to stay in the half-plane $S$}, and provide a justification below.

Observe that $-X_t^\downarrow,X_t^\uparrow\in S\cup\Set{\dagger}$ for $t>0$ a.s., whereas the initial values are classified according to the trichotomy in Lemma~\ref{lem:trichotomy}.
In particular, $X_0^\uparrow=0$ in cases $(\updownarrow),(\uparrow)$, and $X_0^\downarrow=0$ in cases $(\updownarrow),(\downarrow)$.
Importantly, the projected conditioned processes $\inner{X^\uparrow}{\eta}$ and $\inner{X^\downarrow}{\eta}$ coincide with the univariate L\'evy process $Z$ conditioned to stay positive and negative, respectively.
In particular, the lifetimes $\zeta^\uparrow$ and $\zeta^\downarrow$ can be studied using the univariate theory, and so
\[
\zeta^\uparrow=\infty\quad\text{ iff }\quad\limsup_{t\to\infty} Z_t=\infty,\qquad \zeta^\downarrow=\infty\quad\text{ iff }\quad\liminf_{t\to\infty} Z_t=-\infty
\]with probability~$1$. 
Furthermore, $\zeta^\uparrow>0$ unless $Z$ is a non-increasing process and then $\zeta^\uparrow=0$ a.s.
Yet another useful observation is given by the following result.
\begin{lemma}\label{lem:continuity}
The processes $X^\uparrow$ and $X^\downarrow$ do not jump at a fixed $t>0$ a.s.
\end{lemma}
\begin{proof}
Assume that $\underleftarrow X:\p^T$ jumps at $t>0$ with positive probability. Then by an argument as in the proof of Lemma~\ref{lem:trichotomy} we find that we must be in the case~$(\uparrow)$.
Hence $X$ has two jumps separated by time $t$ with positive probability, which is impossible.
According to Theorem~\ref{thm:representation} we find that $X^\downarrow$ has no jump at $t$ a.s.\ when excluding the jump into $\dagger$.
The latter would imply $\p(\zeta^\downarrow=t)>0$, which is again impossible by a similar argument.
By time-reversal the same property is true with respect to $X^\uparrow$.
\end{proof}

Importantly, (under Assumption~\eqref{as:notCPP}) the process $X^\downarrow$ is a.s.\ the same if the non-strict inequalities in its definition are replaced by strict inequalities, which follows from  basic properties of L\'evy processes.
In particular, we find that $X^\downarrow=-(-X)^\uparrow$ a.s. The respective equality in distribution can also be seen using the representation in Theorem~\ref{thm:representation} and the standard time-reversal argument.
Finally, observe a close link to the classical Sparre-Andersen identity~\cite[Lem.\ VI.15]{bertoin}: 
$A^+_T$ has the same law as the time of the supremum of $Z$ on $[0,T]$, which by time-reversal coincides with the law of the lifetime of the respective post-infimum process.

\section{Motivating limit theorem}\label{sec:application}
Bertoin's representation and its above stated generalization are indispensable in the study of L\'evy processes around their extremes.
In the one-dimensional setting it has been fundamental for the results in~\cite{bis_iva,iva_pod}.
We further demonstrate its usefulness by establishing an invariance principle, see~\cite{chaumont_doney_inv} and~\cite{iva_zooming} for alternative approaches in the univariate case (the latter needs a better justification of convergence of Markov processes).
The following short proof requires certain assumptions, and for simplicity we consider only the case of an oscillating $Z_t=\inner{X_t}{\eta}$:
\begin{equation}\label{eq:oscillating}
\limsup_{t\to\infty}Z_t=\infty\qquad \text{and}\qquad \liminf_{t\to\infty}Z_t=-\infty \qquad \text{a.s.}
\end{equation}
Recall that this assumption implies that both $X^\uparrow$ and $X^\downarrow$ have infinite lifetimes.

\begin{theorem}\label{thm:invariance}
Let $X^\n$ be a sequence of L\'evy processes weakly convergent to a L\'evy process $X$ satisfying~\eqref{eq:oscillating} and Assumption~\eqref{as:notCPP}.
Then for any sequence of finite deterministic times $T_n\to \infty$ there is the weak convergence
\[(\underleftarrow X^\n,\underrightarrow X^\n): \p^{T_n}\convd(-X^\downarrow,X^\uparrow).\]
\end{theorem}
\begin{proof}
Fix an arbitrary finite $T>0$. By the continuous mapping theorem we have under $\p^T$:
\[(\underleftarrow X^\n,\underrightarrow X^\n)\convd (\underleftarrow X,\underrightarrow X).\]
Indeed, for converging paths the directional infima and their (right) times must converge assuming the limiting path has no jump at $T$ and it achieves the directional infimum only once (this is a.s.\ true).
Furthermore, $X$ has no jump perpendicular to $\eta$ at $\tau$, see Lemma~\ref{lem:trichotomy} and Figure~\ref{fig:illustration} (right).
Note that making all processes stay at 0 for negative times is essential in the case when the limit process jumps at~$\tau$.

According to Theorem~\ref{thm:representation} we have
\[(-{X^\n}^\downarrow,{X^\n}^\uparrow):\p^T\convd (-X^\downarrow,X^\uparrow):\p^T\]
for every $T>0$, and the latter weakly converges to $(-X^\downarrow,X^\uparrow)$ as $T\to\infty$.
Thus it is left to apply a standard approximation result~\cite[Thm.\ 3.2]{billingsley} or~\cite[Thm.\ 4.28]{kallenberg} to obtain
\begin{equation}\label{eq:convd2}(-{X^\n}^\downarrow,{X^\n}^\uparrow):\p^{T_n}\convd (-X^\downarrow,X^\uparrow),\end{equation}
and hence also the stated result (apply Theorem~\ref{thm:representation} to the left hand side).
The crux of the approximation result consists in showing that the Skorokhod distance (on each compact time interval $[0,t]$) between the left hand side in~\eqref{eq:convd2} and the same object for the time horizon~$T$ converges to 0 in probability  as $T\to\infty$ uniformly for large~$n$.
In our case it is sufficient to check that
\[\lim_{T\to\infty}\limsup_n\p({A^\n}^\pm_{T_n}\wedge {A^\n}^\pm_T>t)= 1,\qquad t>0,\]
where the event corresponds to two identical paths on the time interval~$[0,t]$. 
We may assume that $T_n\geq T$ implying ${A^\n}^\pm_{T_n}\geq {A^\n}^\pm_T$, but the latter weakly converges to $A^\pm_T$. 
Finally, note that~\eqref{eq:oscillating} implies $A^\pm_\infty=\infty$ a.s.
\end{proof}

The above argument can be adapted to include the case where $\lim_{t\to\infty}Z_t=\infty$ and $\lim_{t\to\infty}Z^\n_t=\infty$ for all large enough~$n$, as well as the case with $-\infty$ limits.
 That is, the infinite-time behavior of $Z$ and the approximating sequence $Z^\n$ is the same. 
Otherwise, the proof becomes substantially more difficult and it is then required to work with a compactified space where $\dagger$ is a point at infinity.

Finally, we show that zooming in on $X$ at the time-space location of the directional infimum results in the pair of conditioned processes corresponding to the underlying Brownian part.
This limit law is studied in Proposition~\ref{prop:BM} below.
\begin{corollary}\label{cor:zooming_at_inf}
Let $B$ be the Brownian part of the $d$-dimensional $X$, and assume that $\inner{B_1}{\eta}$ is not a.s.\ zero. Then
\[\sqrt n(\underleftarrow X_{\cdot/n},\underrightarrow X_{\cdot/n}):\p^1  \convd (-B^\downarrow,B^\uparrow).\]
\end{corollary}
\begin{proof}
Define a scaled time-changed process $X^\n_t=\sqrt n X_{t/n}$ and note that $X^\n\convd B$, see~\cite[Prop.\ 2]{bertoin} and~\cite[Thm.\ 15.17]{kallenberg}. 
It is left to apply Theorem~\ref{thm:invariance} with $T_n=n$.
\end{proof}

\section{Linear transformations and the Brownian example}\label{sec:linear}

Linear transformations play an important role in the multivariate theory as demonstrated by the following result.
\begin{lemma}\label{lem:linear_transformation}
	Consider a $d'\times d$ matrix $M$ and $d'$-dimensional vector~$\eta'\neq 0$ such that $M^\top\eta'\neq 0$. Then $(MX)^\uparrow$ defined using $\eta=\eta'$ coincides with $M(X^\uparrow)$ defined using $\eta=M^\top \eta'$.
\end{lemma}
\begin{proof}
Note that $\inner{MX}{\eta'}=\inner{X}{M^\top\eta'}$ and use linearity of the stochastic integral in the definition of~$Y^\pm$.
\end{proof}
Consequently, it suffices to study conditioning for just one direction, say
\[\eta_1=(1,0,\dotsc,0)^\top\in\R^d.\] 
For any unit vector $\eta\in\R^d$  we may choose an orthogonal matrix $R$ ($RR^\top=I$) such that $R\eta=\eta_1$. 
Then $X^\uparrow$ coincides with $R^\top(RX)^\uparrow$ where the latter is defined for the direction $\eta_1$.
Our next result allows us reduce certain multivariate cases to the univariate theory.
\begin{lemma}\label{lem:specialcase}
Consider $X=X'v+X''$, where $X'$ and $X''$ are independent L\'evy processes with dimensions $1$ and $d$ respectively, and additionally $\inner{v}{\eta}>0$,  $\inner{X''_t}{\eta}=0,t\geq 0$.
Then $X^\uparrow\stackrel{d}{=}{X'}^\uparrow v+X''$, where ${X'}^\uparrow$ is the univariate $X'$ conditioned to stay positive and by convention $\dagger\cdot v+x''=\dagger$. 
\end{lemma}
\begin{proof}
Note that the process $\underrightarrow X:\p^T$ has the same law as $\underrightarrow{X'}v+X'':\p^T$, where $\underrightarrow{X'}$ is the post-infimum process of univariate~$X'$.
This is so, because $Z_t=\inner{v}{\eta}X'_t$ and the process $X''$ is independent of~$\tau$, whereas $\underrightarrow{X'v}=\underrightarrow{X'}v$ under $\p^T$. 
It is left to apply Corollary~\ref{cor:object} and the continuous mapping theorem.
\end{proof}

We are now ready to treat the basic example of a conditioned Brownian motion.
In this regard note that a univariate standard Brownian $B^{(1)}$ conditioned to stay positive is a Bessel-3 process which we denote by ${B^{(1)}}^\uparrow$.
\begin{proposition}\label{prop:BM}
Let $X$ be a (driftless) Brownian motion with covariance matrix~$\Sigma$ such that $\Sigma\eta\neq 0$. Then 
\[X^\uparrow\stackrel{d}{=}-X^\downarrow\stackrel{d}{=}MR ({B^{(1)}}^\uparrow,B^{(2)},\ldots,B^{(d)})^\top,\]
where $B=(B^{(1)},\ldots,B^{(d)})^\top$ is a standard Brownian motion in $\R^d$, and the square matrices $M$ and $R$ satisfy 
\[MM^\top=\Sigma,\qquad RR^\top=I,\qquad R^\top M^\top\eta=\sqrt{\eta^\top\Sigma\eta}\eta_1.\]
\end{proposition}
\begin{proof}
The first distributional equality is a consequence of~$-X\stackrel{d}{=}X$. 
Next, using $X\stackrel{d}{=}MRB$ and Lemma~\ref{lem:linear_transformation} we find that $X^\uparrow$ has the law of $MR(B^\uparrow)$ for the direction~$R^\top M^\top\eta$,
where the latter is proportional to $\eta_1$.
It is left to apply Lemma~\ref{lem:specialcase} to find that $B^\uparrow$ for the direction $\eta_1$ has the law of $({B^{(1)}}^\uparrow,B^{(2)},\ldots,B^{(d)})^\top$.
\end{proof}

\begin{example}\label{ex:BM}
Take $d=2, \eta=(a,b)^\top$ and a Brownian motion $X$ with standard deviations $\sigma_1,\sigma_2>0$ and correlation~$\rho\in (-1,1)$. 
Then Proposition~\ref{prop:BM} yields
\[X^\uparrow\stackrel{d}{=}\frac{1}{\sqrt{a^2\sigma_1^2+2ab\sigma_1\sigma_2\rho+b^2\sigma_2^2}}\begin{pmatrix}a\sigma^2_1+b\sigma_1\sigma_2\rho & -b\sigma_1\sigma_2\sqrt{1-\rho^2}\\
a\sigma_1\sigma_2\rho+b\sigma_2^2 & a\sigma_1\sigma_2\sqrt{1-\rho^2}
  \end{pmatrix}\begin{pmatrix}{B^{(1)}}^\uparrow\\ B^{(2)}\end{pmatrix},\]
  where we used a Cholesky square-root~$M$.
\end{example}

In particular, simulation of $X^\uparrow$ over a grid is a trivial task when $X$ is a driftless Brownian motion.
We depict two independent sample paths in Figure \ref{fig:pathsimulation}.
\begin{figure}[H]
	\center
	\includegraphics{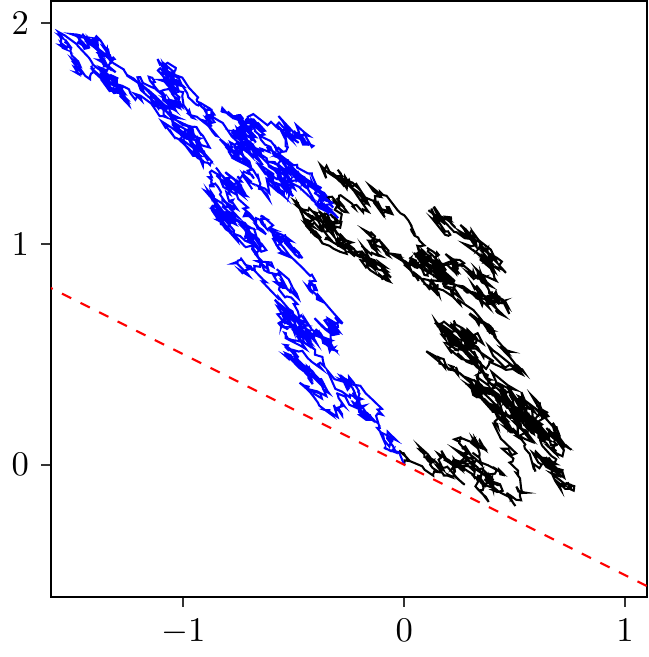}
	\caption{Two independent paths simulated from the common law of $-X^\downarrow$ and $X^\uparrow$ for the direction $\eta=(1,2)^\top$, where $\sigma_1=\sigma_2=1$ and $\rho=-0.8$. The red line is the boundary of the corresponding half-space.}
	\label{fig:pathsimulation}
\end{figure}
Further insight can be obtained from Figure \ref{fig:pointclouds} consisting of three plots, each containing 500 simulations of $X^\uparrow_1$ for different values of $\rho$. 

\begin{figure}[H]
	\centering
	\begin{subfigure}[t]{.33\textwidth}
		\centering
		\includegraphics{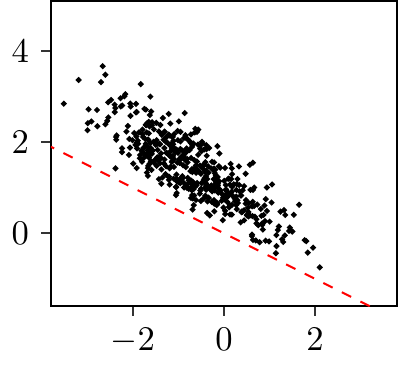}
		\captionsetup{width=.8\linewidth}
		\caption{$\rho=-0.8$.}
		\label{subfig:pointcloudsA}
	\end{subfigure}%
	\begin{subfigure}[t]{.33\textwidth}
		\centering
		\includegraphics{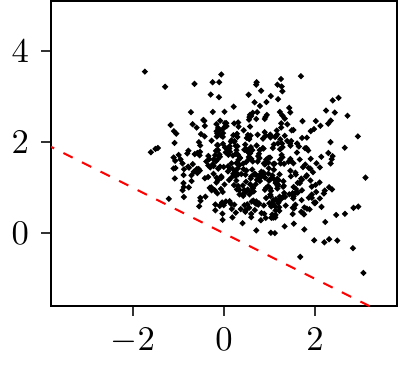}
		\captionsetup{width=.8\linewidth}
		\caption{$\rho=0$.}
		\label{subfig:pointcloudsB}
	\end{subfigure}%
	\begin{subfigure}[t]{.33\textwidth}
		\centering
		\includegraphics{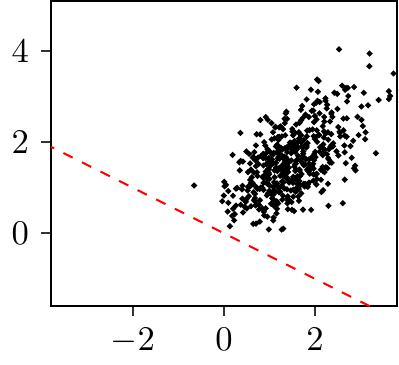}
		\captionsetup{width=.8\linewidth}
		\caption{$\rho=0.8$.}
		\label{subfig:pointcloudsC}
	\end{subfigure}\quad%
	\caption{Simulated values of $X^\uparrow_1$ for the direction $\eta=(1,2)^\top$ and for different values of $\rho$. The standard deviations are $\sigma_1=\sigma_2=1$.}
	\label{fig:pointclouds}
\end{figure}

\section{The law of the limit pair}\label{sec:law}
We need some additional notation.
Consider an extension of the probability space $(D,\mathcal F,\p)$ supporting a standard exponential random variable $e_1$ independent of everything else.
Define $e_q=e_1/q$, an exponential random variable of rate~$q>0$, and let $X:\p^{e_q}$ be the process $X:\p$ killed at~$e_q$.
Finally, the process $X:\p^{e_q}_x$ corresponds to the shifted process $X_t+x\ind{\Set{t\geq 0}}$ killed at~$e_q$, and in the case of no killing we write~$\p_x$.
\subsection{Exponential time horizon}
As in the univariate case the characterization of the law of the limit object $(-X^\downarrow,X^\uparrow)$ proceeds by first studying the respective pair of processes under $\p^{e_q}$,
that is, when the killing time of the original process is an independent exponential random variable of rate~$q>0$.
We start with a simple observation that under $\p^{e_q}$ (and Assumption~\eqref{as:notCPP}) we have
\[(\underleftarrow{X},\underrightarrow{X})\stackrel{d}{=}(\underrightarrow{-X},\underleftarrow{-X}),\]
which readily follows by time-reversal; alternatively one may use Theorem~\ref{thm:representation}. 
The following splitting result is based on some classical arguments, and we only provide appropriate references.

\begin{proposition}\label{prop:decomposition}
	Under $\p^{e_q}$ the processes $\underleftarrow{X}$ and $\underrightarrow{X}$ are independent Markov processes. The semigroup of $\underrightarrow{X}$ is given by
	\begin{equation*}
		\frac{\p^{e_q}_{x}(X_t\in\dd y,\underline Z_t>0)\p_{y}^{e_q}(\underline Z>0)}{\p^{e_q}_{x}(\underline Z>0)}=\p^{e_q}_x(X_t\in \dd y|\underline Z>0),\qquad t>0,x,y\in S.
	\end{equation*}
	Moreover, in case $(\downarrow)$ the initial distribution is given by
	\[\p^{e_q}(\underrightarrow X_0\in \dd y)=\p^{e_q}_y(\underline Z>0)\Pi(\dd y)/\Big(q+\int_{\Set{\inner{z}{\eta}>0}} \p^{e_q}_z(\underline Z>0)\Pi(\dd z)\Big),\quad y\in S.\]
\end{proposition}
\begin{proof}
	The fact that $\underrightarrow{X}$ is Markov with the stated semigroup is proven in \cite{millar_decomposition}.
	Independence of the processes follows by discretizing the local time of $Z$ at its infimum as in the proof of \cite[Lem.~VI.6]{bertoin}. 
	The initial distribution in case $(\downarrow)$ can be obtained analogously to~\cite[Prop.\ 3.3]{iva_splitting} using an enumeration of jumps of~$X$.
\end{proof}

\subsection{Infinite time horizon}
We are now ready to characterize the law of the limit object $(-X^\downarrow,X^\uparrow)$.
Consider a so-called renewal function associated to the ladder height process $\underline H$ corresponding to~$-Z$: 
\[h(x):=\int_0^\infty \p(\underline H_t\leq x)\dd t,\]
where the scaling of local time is arbitrary, see also~\cite[p.\ 157, 171]{bertoin}.
This is exactly the $h$-function appearing in the Doob h-transform corresponding to the univariate $Z$ conditioned to stay positive, see~\cite{bertoin_splitting,chaumont_doney} for alternative representations. 
The function $h$ is finite, continuous and increasing. 

\begin{theorem}\label{thm:chararcterization}
The processes $-X^\downarrow$ and $X^\uparrow$ are independent Markov processes, and the former has the law of $(-X)^\uparrow$. The semigroup of $X^\uparrow$ is given by
	\begin{equation*}
		p^\uparrow_t(x,\dd y):=\frac{h(\inner{y}{\eta})}{h(\inner{x}{\eta})}\p_{x}(X_t\in\dd y,\underline Z_t>0),\qquad t>0,\,x,y\in S.
	\end{equation*}
	Furthermore, in case $(\downarrow)$ and if $Z$ is non-monotone we have
	\begin{equation}\label{eq:initial}\p(X^\uparrow_0\in \dd y)=h(\inner{y}{\eta})\Pi(\dd y)/\int_{\Set{\inner{z}{\eta}>0}} h(\inner{z}{\eta})\Pi(\dd z),\quad y\in S.\end{equation}
\end{theorem}
\begin{proof}
We apply Theorem~\ref{thm:representation} with $T=e_q$ and Proposition~\ref{prop:decomposition}, and then let $q\downarrow 0$.
Since $e_q\to \infty$ we indeed retrieve $-X^\downarrow$ and $X^\uparrow$.
The Markov property follows from the strong convergence result implied by Theorem~\ref{thm:representation} upon recalling that the distribution of $\zeta^\uparrow$ has no atoms, see Lemma~\ref{lem:continuity}.
Let us check that the semigroup in Proposition~\ref{prop:decomposition} has the stated weak limit.
From the univariate theory~\cite[Eq.~(2.5)]{chaumont_doney} we know that for a certain $c_q>0$
\[\p^{e_q}_x(\underline Z>0)/c_q\to h(\inner{x}{\eta}),\qquad x\in S\]
 as $q\downarrow 0$, and it is left to apply the dominated convergence theorem as in~\cite[Prop.\ 1]{chaumont_doney}.

With respect to the initial distribution (for the assumed case) we observe that 
\[\p^{e_q}(\underrightarrow X_0\in A)\to \p(X^\uparrow_0\in A),\qquad \int_A\p^{e_q}_y(\underline Z>0)/c_q\Pi(\dd y)\to \int_Ah(\inner{y}{\eta})\Pi(\dd y)<\infty\] 
for any bounded Borel set $A$, also bounded away from~0 (by the dominated convergence theorem). 
It is left to recall that   $\zeta^\uparrow>0$ and $X_0^\uparrow\in S$ a.s., and $A$ can be chosen so that $\int_Ah(\inner{y}{\eta})\Pi(\dd y)>0$.
The latter is true since $h(z)>0$ for $z>0$ and necessarily $\Pi(S)>0$.
\end{proof}

In the univariate case the initial distribution formula~\eqref{eq:initial} is known in the case of no negative jumps, where $\underline H(t)=t$ implying $h(x)=x$, see~\cite{chaumont_sur} and also \cite[Eq.\ (2.12)]{chaumont_doney}.
Let us also stress the following relation to the univariate conditioned processes. 
\begin{remark}
Choosing the direction $\eta_1=(1,0,\dotsc,0)^\top$ we observe that
\[p^\uparrow_t(x,\dd y)={p_t^{(1)}}^\uparrow (x^{(1)},\dd y^{(1)})\p_{x^{(2:d)}}(X^{(2:d)}_t\in \dd y^{(2:d)}|{X^{(1)}_t=y^{(1)},\underline X^{(1)}_t>0}),\]
where $X=(X^{(1)},X^{(2:d)})$ and ${p_t^{(1)}}^\uparrow$ corresponds to $X^{(1)}$ conditioned to stay positive.
\end{remark}

\section{Starting away from the origin}\label{sec:Feller}
Theorem~\ref{thm:chararcterization} characterizes the law of $X^\uparrow$ in case $(\downarrow)$, but otherwise it lacks convergence of the semigroup as $x\to 0$.
In this section we address this issue and also state a number of further useful properties.
The proofs follow closely the univariate analogues in~\cite{chaumont_doney} and thus we only state the main steps and observations.

It is easy to see that the semigroup $p^\uparrow_t(x,\dd y)$ of $X^\uparrow$, see Theorem~\ref{thm:chararcterization}, is conservative and satisfies the Feller properties on~$\hat S:=S\cup\Set{\dagger}$.
Note that the hyperplane defining this half-space has been excluded.
We write $X:P_{x}^\uparrow$ for the respective Feller process indexed by $[0,\infty)$ and started at $x\in \hat S$, and note that it satisfies the strong Markov property~\cite[Thm.\ 19.17]{kallenberg}.

Observe that the law of the Markov process with the semigroup in Proposition~\ref{prop:decomposition} when started in $x\in \hat S$ can be conveniently written as
\begin{equation}\label{eq:conditioning}
X|\underline Z>0=X|X_t\in \hat S\; \forall t\geq 0,\qquad\text{under }\p^{e_q}_x.
\end{equation}
Furthermore, \cite[Prop.\ 1]{chaumont_doney} readily generalizes to 
\begin{equation}\label{eq:limit}
P_x^\uparrow(\Lambda, t<\zeta)=\lim_{q\downarrow 0}\p_x^{e_q}(\Lambda, t<\zeta|\underline Z>0),\qquad \Lambda\in\mathcal F_t,t>0,x\in S,
\end{equation}
which explains the name `conditioned to stay in a half-space'.
Note that $\inner{X}{\eta}:P^\uparrow_x$ is the univariate process $\inner{X}{\eta}$ conditioned to stay positive and started from $\inner{x}{\eta}$.

\begin{proposition}\label{prop:markov_repr}
For any $x\in S$  the process $Z_t=\inner{X_t}{\eta}$ under $P_x^\uparrow$ has a unique and finite time of infimum, $\underrightarrow X$ and $\underleftarrow X$ are independent under $P^\uparrow_x$, and 
\[\underrightarrow X:P^\uparrow_x\stackrel{d}{=} X^\uparrow.\]
Furthermore,
	\[X:P^\uparrow_x \convd X^\uparrow,\qquad \text{ as }S\ni x\to 0,\]
	where by convention the sample paths satisfy $\omega_t=0,t<0$.
\end{proposition}
\begin{proof}
It follows from the calculations in~\cite[p.\ 956]{chaumont_doney} that the time of infimum is finite. 
Consider the process in~\eqref{eq:conditioning} and establish a splitting result analogous to Proposition~\ref{prop:decomposition}.
The post-infimum process has the law of $\underrightarrow X:\p^{e_q}$, and so we can apply \eqref{eq:limit} and Theorem~\ref{thm:representation} to get the first statement.

In view of the first part, it is only required to show that the pre-infimum process $\underleftarrow X:P^\uparrow_x$ becomes negligible in probability as $x\to 0$.
The arguments of~\cite[Thm.\ 2]{chaumont_doney} still apply, and we additionally show that the maximal fluctuation of the pre-limit process perpendicular to $\eta$ is negligible.
This can be done by considering the stopping time $\nu=\inf\Set{t\geq 0\given \norm{X_t-x}>\epsilon}$ and employing similar analysis based on the strong Markov property.
In case $(\updownarrow)$ and $(\downarrow)$ we then need to show that $\p_x(\underline Z_\nu>0)\to 0$, which is indeed true.
\end{proof}

The above proof, in fact, shows that 
\[X:P^\uparrow_x \convd x_0+X^\uparrow,\qquad \text{ as }S\ni x\to x_0,\inner{x_0}{\eta}=0.\]
In cases $(\updownarrow),(\uparrow)$ the process $x_0+X^\uparrow$ starts at $x_0$ and according to Lemma~\ref{lem:continuity} it does not jump at fixed times.
Hence in these cases we may extend our Feller process $X:P^\uparrow_x$ to the state space $\overline S\cup\Set{\dagger}$, the closed half-space with an absorbing state,
by setting \[X:P^\uparrow_{x_0}:=x_0+X^\uparrow\qquad\text{ for any }x_0\text{ with }\inner{x_0}{\eta}=0.\]
Note that this definition coincides with the result of the construction presented in Section~\ref{sec:representation} 
if we take $X$ started at $x_0$ and let $Y_t^+=-\int_{[-t,0]}\ind{\Set{\inner{\tilde X_{s-}}{\eta}\geq 0}}\idd\tilde X_s$ which yields an a.s.\ identical process in the original case.

\section{Conjecture: zooming in at the maximal distance from the origin}\label{sec:further_applications}

For a possible further application we turn our attention to the local behavior of a Lévy process at the time when it reaches the maximal distance from the origin.
Assuming finite life time, $\zeta\in(0,\infty)$, we let 
\[\tau:=\sup\Set{t\geq 0\given\norm{X_t}\vee\norm{X_{t-}}=\sup_{s\geq0}\norm{X_s}}\in[0,\zeta]\] be the (last) time when the Euclidean norm is maximal. Consider the respective position $M:=X_\tau$ if $\norm{X_\tau}\geq\norm{X_{\tau-}}$ and $M:=X_{\tau-}$ otherwise, and define the processes
\begin{equation*}
	\overrightarrow{X}_t:=\begin{cases}
		X_{\tau+t}-M &\text{if }0\leq t<\zeta-\tau \\
		\dagger &\text{if }t\geq\zeta-\tau,
	\end{cases}
	\qquad
	\overleftarrow{X}_t:=\begin{cases}
		X_{(\tau-t)-}-M &\text{if }0\leq t<\tau \\
		\dagger &\text{if }t\geq\tau.
	\end{cases}
\end{equation*}
Observe that the pair $(\overleftarrow{X},\overrightarrow{X})$ coincides with $(\underleftarrow{X},\underrightarrow{X})$ studied above for the (path-dependent) direction $\eta=-M$, see
also Figure~\ref{fig:circle} for a schematic illustration.
\begin{figure}[H]
\begin{centering}
\begin{tikzpicture}
\draw[dashed] (0,0) circle (2);
\filldraw (0,0) circle (1pt) node[anchor=west]{0};
\filldraw (-0.894,-1.789) circle (1pt)  node[anchor=north]{$M$};
\draw[dotted] (-0.894,-1.789) circle (15pt);
\draw[gray] (-0.894+2,-1.789-1) --(-0.894-2,-1.789+1);
\draw[thick] plot [smooth, tension=1] coordinates { (0,0) (-0.3,-0.3) (-1,1) (-0.2,0.3) (-0.9,-0.7) (-0.894,-1.789)};
\draw[thick] plot [smooth, tension=1] coordinates { (-0.894,-1.789) (-0.5,-1.5) (-0.3,-1.8) (-0.6,-1.7) (0.2,-1.5)};
\filldraw (0.2,-1.5) circle (1pt);
\draw[gray,->] (-0.894,-1.789)--(-0.447,-0.894) node[anchor=east]{$\eta$};
\end{tikzpicture}
\end{centering}
\caption{Schematic illustration of zooming in at the maximal distance.}
\label{fig:circle}
\end{figure}
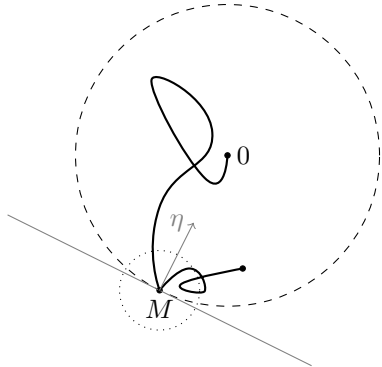
Inspired by Corollary~\ref{cor:zooming_at_inf} and using the intuition from the one-dimensional stable convergence in~\cite{iva_zooming}, we conjecture the following result; proving it seems to be exceedingly challenging at the moment. We anticipate that the convergence is again stable~\cite{aldous} but avoid complicating the statement. 

\begin{conjecture}\label{thm:zooming_max_norm}
	Let $B$ be the Brownian part of $X$ with a non-singular covariance matrix. Then
	\[\sqrt{n}(\overleftarrow{X}_{\cdot/n},\overrightarrow{X}_{\cdot/n}):\p^1\convd (-B^\Downarrow, B^\Uparrow), \] where the limit pair is a mixture of $(-B^{\downarrow},B^{\uparrow})$ for the independent direction $\eta=-M:\p^1$. 
\end{conjecture}

We illustrate this conjecture by a simulation study where $X=B$ is a 2-dimensional Brownian motion with correlation $\rho=-0.8$ as in Section~\ref{sec:linear}. We simulate $K$ (approximate) copies of the random vector $\sqrt{n}\overrightarrow{X}_{1/n}$ under $\p^1$ for $n=1000$ using discretization with step size~$10^{-5}$. The $K$ samples of the limit quantity $B^\Uparrow_1$ are constructed by reusing the directions $\eta=-M:\p^1$ and then independently sampling $B^\uparrow_1$ according to Example~\ref{ex:BM}.

Note that $\overrightarrow{X}$ may have a lifetime strictly smaller than $1/n$, making $\sqrt{n}\overrightarrow{X}_{1/n}$ undefined. In our simulation we exclude these cases, effectively conditioning $\overrightarrow{X}$ to have a lifetime larger than $1/n$. We simulated 5000 times, resulting in $K=4833$ (conditional) samples of $\sqrt{n}\overrightarrow{X}_{1/n}$.
The respective bivariate densities are presented in Figure \ref{fig:zooming_at_max_norm_density}, and we observe that they are indeed rather close.

\begin{figure}[H]
	\adjustbox{center}{
	\clipbox{.3cm 1.5cm 0cm 1cm}{
		\includegraphics{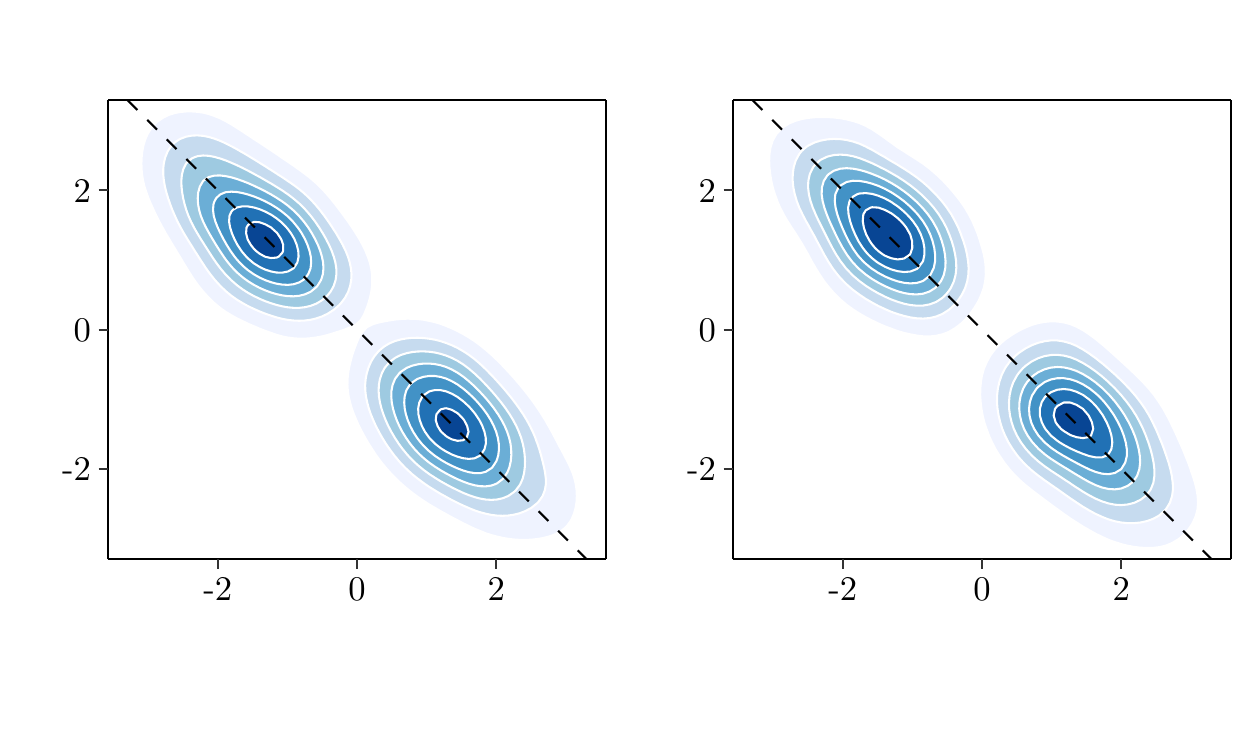}
	}
	}
	\caption{Estimated bivariate densities for $\sqrt{n}\protect\overrightarrow{X}_{1/n}$ and $B_1^\Uparrow$ on the left and right respectively. A darker shade of blue indicates a higher density. The dashed line is the line through the origin with slope $-1$.}
	\label{fig:zooming_at_max_norm_density}
\end{figure}

It is noted that the Brownian motion $X$ with correlation $\rho$ tends to achieve its maximal distance from the origin in the NW or SE direction, which leads to the two clusters of points in Figure~\ref{fig:zooming_at_max_norm_density}. 

\appendix
\section{Proof of Theorem \ref{thm:representation}}\label{sec:representation_proof}

\subsection{Discrete time}\label{subsec:representation_discrete}

We begin by stating a discrete-time version of Theorem \ref{thm:representation}. 
Fix $\zeta\in\N$ and consider a process $X\in\R^d$ over the index set $\{0,\ldots,\zeta\}$ together with the projected process $Z_i:=\inner{X_i}{\eta}$. 
Let $\tau:=\sup\Set{i\leq \zeta\given Z_i=\underline{Z}_i}$ be the index of the (last) minimum of $Z$, and $\underline X:=X_\tau$ be the value of the directional minimum. The directional post-minimum and reversed pre-minimum chains $\underrightarrow{X}$ and $\underleftarrow{X}$ are given by
\begin{equation*}
	\underrightarrow{X}_i:=\begin{cases}
		X_{\tau+i}-\underline X &\text{if }i\leq \zeta-\tau \\
		\dagger &\text{if }i>\zeta-\tau,
	\end{cases}
	\qquad
	\underleftarrow{X}_i:=\begin{cases}
		X_{\tau-i}-\underline X &\text{if }i\leq\tau \\
		\dagger &\text{if }i>\tau.
	\end{cases}
\end{equation*}

Next, define
\begin{equation}\label{eq:A_pm}
	A_i^+:=\sum_{j=1}^i\ind{\Set{Z_j>0}},\qquad A_i^-:=\sum_{j=1}^i\ind{\Set{Z_j\leq0}}
\end{equation}
when $i\leq \zeta$, and let $\alpha_i^\pm:=\inf\Set{j\in\N\given A_j^\pm=i}$ denote the inverses of $A^\pm$. With $\Delta X_j:=X_j-X_{j-1}$ we define the chains $X^\uparrow$ and $X^\downarrow$ by
\begin{equation*}
	X^\uparrow_i:=\begin{cases}
		\sum_{j=1}^{\alpha_i^+}\ind{\Set{Z_j>0}}\Delta X_j &\text{if }i\leq A_\zeta^+ \\
		\dagger & \text{if }i>A_\zeta^+,
	\end{cases}
	\qquad
	X^\downarrow_i:=\begin{cases}
		\sum_{j=1}^{\alpha_i^-}\ind{\Set{Z_j\leq0}}\Delta X_j &\text{if }i\leq A_\zeta^- \\
		\dagger & \text{if }i>A_\zeta^-.
	\end{cases}
\end{equation*}
We are now ready to state the discrete analogue of Theorem~\ref{thm:representation}.
\begin{theorem}\label{thm:representation_discrete}
Assume that $\zeta\in\N$ and $X$ has exchangeable increments.
	Then the pairs of processes $(X^\downarrow,X^\uparrow)$ and $(-\underleftarrow{X},\underrightarrow{X})$ have the same law.
\end{theorem}
\begin{proof}
	The proof of \cite[Thm.~2.1]{bertoin_splitting} is easily adapted to this setting.
\end{proof}

\subsection{Continuous time}

The proof of Theorem \ref{thm:representation} proceeds much like the proof of \cite[Thm.~3.1]{bertoin_splitting}. We discretize, apply Theorem \ref{thm:representation_discrete} and take the limit.

Recall that we are considering a Lévy process $X:\p^T$ up to a finite time horizon~$T>0$. For each $n\in\N$ let $X^n$ be the chain given by $X^n_i:=X_{i/n}$, and let $X^{n\uparrow}$ and $X^{n\downarrow}$ be the chains obtained from $X^n$ by the procedure in \S\ref{subsec:representation_discrete}. Define
\begin{equation*}
	Y^{n+}_i:=\sum_{j=1}^{i}\ind{\Set{\inner{X^n_j}{\eta}>0}}\Delta X^n_j,
\end{equation*}
and note that almost surely
\begin{equation*}
	Y^{n+}_{[tn]}=-\sum_{i=-[tn]}^{-1}\ind{\Set{\inner{\tilde X_{i/n}}{\eta}>0}}(\tilde X_{(i+1)/n}-\tilde X_{i/n}).
\end{equation*}
By \cite[Cor.~17.13]{kallenberg} we have
\begin{equation*}
	\sup_{0\leq t\leq T}\norm{Y^{n+}_{[tn]}-Y^+_t}\overset{\p}{\rightarrow}0.
\end{equation*}

Consider further the increasing chains $A^{n\pm}$ obtained from $X^n$ through the construction in \eqref{eq:A_pm}, and let $\alpha^{n\pm}$ be the inverses. Note that $\frac{1}{n}A^{n+}_{[tn]}\to A^+_t$ for all $t\geq0$ a.s.\ since the zero set of $\inner{X_\cdot}{\eta}$ is a Lebesgue null-set a.s. It follows that almost surely $\frac{1}{n}\alpha^{n+}_{[tn]}\to\alpha^+_t$ for all $t\in\cp(\alpha^+)$, where $\cp(\alpha^+)$ is the set of continuity points for $\alpha^+$. To see this, observe first that
\begin{equation*}
	\inf\Set{s\geq0\given\tfrac{1}{n}A^{n+}_{[sn]}>t}=\inf\Set{s\geq0\given A^{n+}_{[sn]}\geq[tn]+1}=\frac{1}{n}\alpha^{n+}_{[tn]+1}.
\end{equation*}
Almost surely the expression on the left converges to $\alpha^+_t$ for all $t\in\cp(\alpha^+)$. This basic convergence of right-continuous inverses is easy to  prove (e.g.\ using the arguments in the proof of \cite[Prop.~0.1]{resnick2013extreme}). Lastly one verifies that $\frac{1}{n}\alpha^{n+}_{[tn]+1}$ can indeed by replaced by $\frac{1}{n}\alpha^{n+}_{[tn]}$.

Let $f\colon[0,\infty)\to\R^d$ be a continuous function.
Then it follows from the observations above that
\begin{equation*}
	\frac{1}{n}\sum_{i=0}^{[Tn]}\inner{f(i/n)}{X^{n\uparrow}_i}\cip\int_0^T \inner{f(s)}{X^\uparrow_s}\idd s,
\end{equation*}
where we make the convention that $\inner{a}{\dagger}=\infty$ for $a\neq0$ and $\inner{0}{\dagger}=0$. To prove this we use the fact that $\alpha^+$ is strictly increasing and has at most countably many discontinuities, with the former implying that $Y^+$ jumps at $\alpha^+_t$ for at most countably many $t$.

Similarly, if $g\colon[0,\infty)\to\R^d$ is a continuous function we obtain the convergence
\begin{equation*}
	\frac{1}{n}\sum_{i=0}^{[Tn]}\inner{g(i/n)}{X^{n\downarrow}_i}\cip\int_0^T \inner{g(s)}{X^\downarrow_s}\idd s.
\end{equation*}

Using the fact that almost surely $(X_t)_{t\in[0,T]}$ reaches its infimum in the direction given by $\eta$ exactly once, it follows that almost surely
\begin{align*}
	\frac{1}{n}\sum_{i=0}^{[Tn]}\inner{f(i/n)}{\underrightarrow{X}^n_i}&\to\int_0^T\inner{f(s)}{\underrightarrow{X}_s}\idd s, \\
	\intertext{and}
	\frac{1}{n}\sum_{i=0}^{[Tn]}\inner{g(i/n)}{\underleftarrow{X}^n_i}&\to\int_0^T\inner{g(s)}{\underleftarrow{X}_s}\idd s,
\end{align*}
where $f$ and $g$ are as above. By Theorem \ref{thm:representation_discrete} we obtain the distributional identity
\begin{equation*}
	\left(\int_0^T \inner{g(s)}{X^\downarrow_s}\idd s,\int_0^T \inner{f(s)}{X^\uparrow_s}\idd s\right)\overset{d}{=}\left(-\int_0^T\inner{g(s)}{\underleftarrow{X}_s}\idd s,\int_0^T\inner{f(s)}{\underrightarrow{X}_s}\idd s\right)
\end{equation*}
under $\p^T$, thus proving Theorem \ref{thm:representation}. \qed

\section*{Acknowledgements}
The authors gratefully acknowledge financial support of Sapere Aude Starting Grant 8049-00021B ``Distributional Robustness in Assessment of Extreme Risk'' from Independent Research Fund Denmark.


\begin{thebibliography}{10}

\bibitem{aldous}
D.~J. Aldous and G.~K. Eagleson, \emph{On mixing and stability of limit
  theorems}, Ann. Probability \textbf{6} (1978), no.~2, 325--331. \MR{517416}

\bibitem{AGP}
S.~Asmussen, P.~Glynn, and J.~Pitman, \emph{Discretization error in simulation
  of one-dimensional reflecting {B}rownian motion}, Ann. Appl. Probab.
  \textbf{5} (1995), no.~4, 875--896. \MR{1384357}

\bibitem{asmussen2018discretization}
S.~Asmussen and J.~Ivanovs, \emph{Discretization error for a two-sided
  reflected {L}\'{e}vy process}, Queueing Syst. \textbf{89} (2018), no.~1-2,
  199--212. \MR{3803961}

\bibitem{bertoin_splitting}
J.~Bertoin, \emph{Splitting at the infimum and excursions in half-lines for
  random walks and {L}\'{e}vy processes}, Stochastic Process. Appl. \textbf{47}
  (1993), no.~1, 17--35. \MR{1232850}

\bibitem{bertoin}
J.~Bertoin, \emph{L\'{e}vy processes}, Cambridge Tracts in Mathematics, vol. 121,
  Cambridge University Press, Cambridge, 1996. \MR{1406564}

\bibitem{billingsley}
P.~Billingsley, \emph{Convergence of probability measures}, second ed., Wiley
  Series in Probability and Statistics: Probability and Statistics, John Wiley
  \& Sons, Inc., New York, 1999, A Wiley-Interscience Publication. \MR{1700749}

\bibitem{bis_iva}
K.~Bisewski and J.~Ivanovs, \emph{Zooming-in on a {L}\'{e}vy process: failure
  to observe threshold exceedance over a dense grid}, Electron. J. Probab.
  \textbf{25} (2020), Paper No. 113, 33. \MR{4161123}

\bibitem{chaumont_sur}
L.~Chaumont, \emph{Sur certains processus de {L}\'{e}vy conditionn\'{e}s \`a
  rester positifs}, Stochastics Stochastics Rep. \textbf{47} (1994), no.~1-2,
  1--20. \MR{1787140}

\bibitem{chaumont_doney}
L.~Chaumont and R.~A. Doney, \emph{On {L}\'{e}vy processes conditioned to stay
  positive}, Electron. J. Probab. \textbf{10} (2005), no. 28, 948--961.
  \MR{2164035}

\bibitem{chaumont_doney_inv}
L.~Chaumont and R.~A. Doney, \emph{Invariance principles for local times at the maximum of random
  walks and {L}\'{e}vy processes}, Ann. Probab. \textbf{38} (2010), no.~4,
  1368--1389. \MR{2663630}

\bibitem{doring_watson_weissmann}
L.~D\"{o}ring, A.~R. Watson, and P.~Weissmann, \emph{L\'{e}vy processes with
  finite variance conditioned to avoid an interval}, Electron. J. Probab.
  \textbf{24} (2019), Paper No. 55, 32. \MR{3968717}

\bibitem{doring_weissmann}
L.~D\"{o}ring and P.~Weissmann, \emph{Stable processes conditioned to hit an
  interval continuously from the outside}, Bernoulli \textbf{26} (2020), no.~2,
  980--1015. \MR{4058358}

\bibitem{duquesne}
T.~Duquesne, \emph{Path decompositions for real {L}evy processes}, Ann. Inst.
  H. Poincar\'{e} Probab. Statist. \textbf{39} (2003), no.~2, 339--370.
  \MR{1962781}

\bibitem{feller}
W.~Feller, \emph{An introduction to probability theory and its applications.
  {V}ol. {II}}, second ed., John Wiley \& Sons, Inc., New York-London-Sydney,
  1971. \MR{0270403}

\bibitem{iva_splitting}
J.~Ivanovs, \emph{Splitting and time reversal for {M}arkov additive processes},
  Stochastic Process. Appl. \textbf{127} (2017), no.~8, 2699--2724.
  \MR{3660888}

\bibitem{iva_zooming}
J.~Ivanovs, \emph{Zooming in on a {L}\'{e}vy process at its supremum}, Ann. Appl.
  Probab. \textbf{28} (2018), no.~2, 912--940. \MR{3784492}

\bibitem{iva_pod}
J.~Ivanovs and M.~Podolskij, \emph{Optimal estimation of some random quantities
  of a {L}\'{e}vy process}, arXiv preprint arXiv:2001.02517 (2020).

\bibitem{kallenberg}
O.~Kallenberg, \emph{Foundations of modern probability}, second ed.,
  Probability and its Applications (New York), Springer-Verlag, New York, 2002.
  \MR{1876169}

\bibitem{millar_decomposition}
P.~W. Millar, \emph{A path decomposition for {M}arkov processes}, Ann.
  Probability \textbf{6} (1978), no.~2, 345--348. \MR{461678}

\bibitem{resnick2013extreme}
S.~I. Resnick, \emph{Extreme values, regular variation, and point processes},
  Applied Probability. A Series of the Applied Probability Trust, vol.~4,
  Springer-Verlag, New York, 1987. \MR{900810}

\end{thebibliography}
\end{document}